\documentclass[final,a4paper,reqno,11pt]{amsart}
\usepackage[final]{graphicx}
\usepackage[final]{hyperref}
\usepackage{amssymb, amstext, amsmath, amsthm, amsfonts, mathrsfs, enumitem, lmodern, ifdraft, stmaryrd}
\usepackage{tikz}
\usepackage{tikz-cd}
\usetikzlibrary{decorations.pathmorphing}
\usepackage[noadjust]{cite}
\usepackage[utf8]{inputenc}
\usepackage[english]{babel}
\usepackage[noabbrev]{cleveref}
\usepackage{wasysym}
\usepackage{float}
\usepackage[activate={true,nocompatibility},final,tracking=true,kerning=true,spacing=true,factor=1100,stretch=10,shrink=10]{microtype}
\microtypecontext{spacing=nonfrench}

\newtheorem{theorem}{Theorem}[section]

\newtheorem{lemma}[theorem]{Lemma}

\newtheorem*{theorem*}{Theorem}
\newtheorem*{corollary*}{Corollary}
\newtheorem{introtheorem}{Theorem}

\crefname{introcorollary}{corollary}{corollary}

\theoremstyle{definition}
\newtheorem{definition}[theorem]{Definition}

\newtheorem*{conjecture*}{Conjecture}

\theoremstyle{remark}

\newtheorem*{acknowledgment}{Acknowledgment}

\newcommand{\sA}{\mathscr{A}}

\newcommand{\sF}{\mathscr{F}}
\newcommand{\sE}{\mathscr{E}}

\newcommand{\add}{\mathsf{add}\hspace{.01in}}

\newcommand{\Ab}{\mathsf{Ab}}

\newcommand{\End}{\operatorname{End}\nolimits}
\newcommand{\Aut}{\operatorname{Aut}\nolimits}

\newcommand{\YExt}{\operatorname{YExt}\nolimits}

\newcommand{\Hom}{\operatorname{Hom}\nolimits}
\newcommand{\id}{\operatorname{id}\nolimits}

\newcommand{\op}{\operatorname{op}\nolimits}

\newcommand{\rad}{\operatorname{rad}\nolimits}

\newcommand{\triv}{\operatorname{triv}}

\renewcommand{\subset}{\subseteq}
\makeatletter
\def\Biggg#1{{\hbox{$\left#1\vbox to25\p@{}\right.\n@space$}}}
\makeatother

\newcommand{\includepdf}[1]{
   \begin{center}
      \includegraphics{#1}
   \end{center}
}
\newcommand{\includepdfdot}[1]{
   \begin{center}
      \includegraphics{#1}.
   \end{center}
}
\newcommand{\includepdfcomma}[1]{
   \begin{center}
      \includegraphics{#1},
   \end{center}
}

\begin{document}
   \title[$n$-Exact categories arising from $(n+2)$-angulated categories]{$n$-Exact categories arising from $(n+2)$-angulated categories}

   \author{Carlo Klapproth}\address{Department of Mathematics, Aarhus University, 8000 Aarhus C, Denmark} 
   \email{carlo.klapproth@math.au.dk}
   
   \begin{abstract}
      Let $(\sF,\Sigma_n,\pentagon)$ be an $(n+2)$-angulated Krull-Schmidt category and $\sA \subset \sF$ an $n$-extension closed, additive and full subcategory with $\Hom_{\sF}(\Sigma_n \sA, \sA) = 0$.
      Then $\sA$ naturally carries the structure $(\sA, \sE_{\sA})$ of an $n$-exact category in the sense of \cite[definition 4.2]{jasso_n-abelian_2016}, arising from short $(n+2)$-angles in $(\sF,\Sigma_n,\pentagon)$ with objects in $\sA$ and there is a binatural and bilinear isomorphism $\YExt^{n}_{(\sA,\sE_{\sA})}(A_{n+1},A_0) \cong \Hom_{\sF}(A_{n+1}, \Sigma_n A_{0})$ for $A_0, A_{n+1} \in \sA$.
      For $n = 1$ this has been shown in \cite{dyer_exact_2005} and we generalize this result to the case $n > 1$.
      On the journey to this result, we also develop a technique for harvesting information from the higher octahedral axiom (N4*) as defined in \cite[section 4]{bergh_axioms_2013}.
      Additionally, we show that the axiom (F3) for pre-$(n+2)$-angulated categories, stating that a commutative square can be extended to a morphism of $(n+2)$-angles and defined in \cite[definition 2.1]{geiss_n-angulated_2013}, implies a stronger version of itself.
   \end{abstract}

   \maketitle
   \section*{Introduction}
   Let $\sA$ be an additive extension-closed subcategory of a triangulated category $(\sF, \Sigma, \vartriangle)$, with $\Hom_{\sF}(\Sigma \sA, \sA) = 0$.
   Then $\sA$ carries the structure $(\sA, \sE_{\sA})$ of an exact category, where the conflations in $\sA$ are precisely the sequences $A_0 \to A_1 \to A_2$ which can be completed to a triangle $A_0 \to A_1 \to A_2 \rightsquigarrow$ in $\vartriangle$.
   Moreover, there is an isomorphism $\YExt^1_{(\sA,\sE_{\sA})}(A_2, A_0) \cong \Hom_{\sF}(A_2, \Sigma A_0)$ which is natural in $A_0,A_2 \in \sA$.
   This is the main theorem of \cite{dyer_exact_2005} and it was also rediscovered in \cite[proposition 2.5]{joergensen_simple_minded_2020}. 
   We show that this result has a higher counterpart:
   \begin{introtheorem} \label{introtheorem:maintheorem}
      Let $\sA$ be an full, additive and $n$-extension closed subcategory of an $(n+2)$-angulated Krull-Schmidt category $(\sF, \Sigma_n, \text{\upshape \pentagon})$ with $\Hom_{\sF}(\Sigma_n \sA, \sA) = 0$.
      Then 
      \begin{enumerate}
         \item $(\sA, \sE_{\sA})$ is an $n$-exact category and
         \item there is a natural and bilinear isomorphism $\Hom_{\sF}(-, \Sigma_n(-)) \to \YExt^n_{(\sA, \sE_{\sA})}(-, -)$ of functors $\sA \times \sA^{\op} \to \Ab$. \label{enum:introthm-second-point}
      \end{enumerate}
   \end{introtheorem}
   Here $n$-exact categories are as defined in \cite[definition 4.2]{jasso_n-abelian_2016} and $\YExt^n(A_{n+1}, A_0)$ denotes the $n$-Yoneda-extensions of $A_{n+1}$ by $A_0$ in $(\sA, \sE_{\sA})$, that is the class of all sequences $A_0 \to A_1 \to \cdots \to A_n \to A_{n+1}$ in $\sE_{\sA}$ modulo the equivalence relation generated by the weak equivalences of $n$-exact sequences as defined in \cite[definition 2.9]{jasso_n-abelian_2016}.
   The $n$-exact structure $\sE_{\sA}$ is defined analogously to the theorem for triangulated categories.
   The proof of \cref{introtheorem:maintheorem} has some new features.
   Notably, showing that the composition of two inflations is an inflation (see \cite[definition 4.2(E1)]{jasso_n-abelian_2016}) causes a problem as the (very technical) higher octahedral axiom \cite[definition 2.1(F4)]{geiss_n-angulated_2013} is required for this.

   Hence, we reintroduce minimal $(n+2)$-angles (see \cite[lemma 5.18]{oppermann_higher-dimensional_2012} and \cite[lemma 3.14]{Fedele_2019}), essentially assigning an, up to isomorphism and rotation, unique $(n+2)$-angle to a given morphism.
   We then use this and the version of the higher octahedral axiom (N4*) defined in \cite[section 4]{bergh_axioms_2013} to show that the objects of the minimal $(n+2)$-angle of a composite $h_0 = f_0 g_0$ can be calculated by means of the objects of any $(n+2)$-angle 
   \includepdf{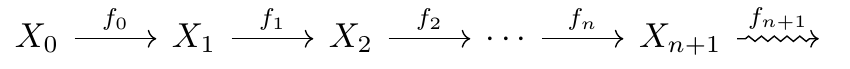}
   containing $f_0$, the objects of any $(n+2)$-angle 
   \includepdf{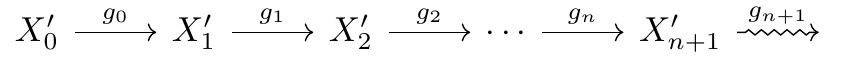}
   containing $g_0$ and the minimal $(n+2)$-angle of $g_{n+1} \Sigma_n f_1$ (see \cref{lemma:comparison-lemma}).
   This will equip us with the necessary leverage to show that \cite[definition 4.2(E1)]{jasso_n-abelian_2016} holds in the situation of \cref{introtheorem:maintheorem} and it is likely very useful in other situations as well.

   On the journey of showing part (\ref{enum:introthm-second-point}) of \cref{introtheorem:maintheorem} we also show that the axiom (F3) from \cite[definition 2.1]{geiss_n-angulated_2013} implies a stronger version of itself:
   Suppose we are given a commutative diagram
   \includepdf{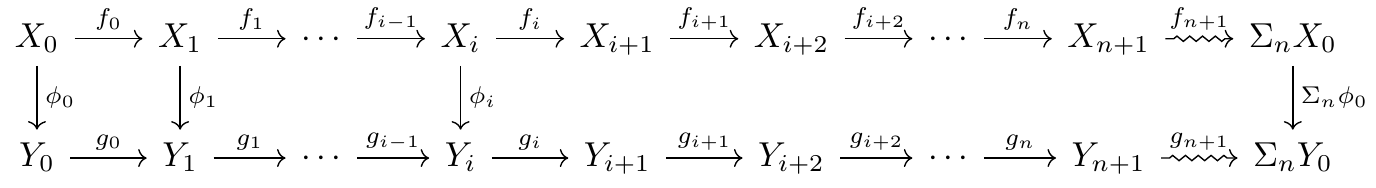}
   where the upper row is an $(n+2)$-angle $X$ and the lower row is an $(n+2)$-angle $Y$ and $i \geq 2$ is arbitrary.
   We show that we can complete the given morphisms to a morphism $\phi = (\phi_0, \dots, \phi_{n+1}) \colon X \to Y$ of $(n+2)$-angles.
   This will be used to show that the connecting morphism $f_{n+1}$ of $X$ is uniquely determined by $f_0, \dots, f_n$ if $\Hom_{\sF}(\Sigma_n X_0, X_{n+1}) = 0$ generalizing [BBD82, corollaire 1.1.10(ii)] to $(n+2)$-angulated categories (see \cref{lemma:connecting-morphism-is-unique}).

   \section{Preliminaries and notations} \label{sec:preliminaries}

   Given morphisms $f \colon X \to X'$ and $g \colon X' \to X''$ we denote the composite $X \xrightarrow{f} X' \xrightarrow{g} X''$ by $fg$.
   Throughout this paper suppose that $(\sF,\Sigma_n,\pentagon)$ is pre-$(n+2)$-angulated category in the sense of \cite[definition 2.1]{geiss_n-angulated_2013}.
   As in \cite{geiss_n-angulated_2013} we assume that $\Sigma_n$ is an automorphism of $\sF$ and not just an autoequivalence.
   When no ambient category is mentioned we implicitly assume it to be $\sF$.
   Further, all subcategories considered will be full subcategories, unless mentioned otherwise.
   Likewise to \cite[definition 2.1]{geiss_n-angulated_2013} we call a sequence of morphisms $X_0 \to X_1 \to \cdots \to X_{n+1} \to \Sigma_n X_0$ an $n$-$\Sigma_n$-sequence.
   Such an $n$-$\Sigma_n$-sequence is called $(n+2)$-angle if it belongs to $\pentagon$.
   In that case we may write the rightmost morphism squiggly, to emphasize that this $n$-$\Sigma_n$-sequence is a $(n+2)$-angle.
   We then may also drop writing the target $\Sigma_n X_0$ of the rightmost morphism.
   Suppose $X$ is an $(n+2)$-angle of the shape
   \includepdfdot{tikz/n-plus-2-angle-generic.pdf}
   Then we call $f_i$ the morphism in \emph{position $i$} of $X$.
   As we will often deal with trivial $(n+2)$-angles, we define a map $\triv \colon \sF \times \{0, \dots, n+1\} \to \pentagon$ assigning to $(Z,i) \in \sF \times \{0, \dots, n+1\}$ the trivial $(n+2)$-angle $0 \to 0 \to \cdots \to 0 \to Z \to Z \to 0 \to \cdots \to 0 \to 0 \rightsquigarrow$
   having $\id_Z$ as the morphism in position $i$.
   We will always use subscript notation, i.e.\ we will write $\triv_i(Z) := \triv(Z,i)$ for $Z \in \sF$ and $i \in \{0,\dots,n+1\}$.

   For convenience we restate the following definitions:
   \begin{definition} \label{definition:n-extension-closed}
      A subcategory $\sA \subset \sF$ is called 
      \begin{enumerate}
         \item \emph{additive} if it is closed under direct summands and direct sums and 
         \item \emph{$n$-extension closed} if for all morphisms $f_{n+1} \colon A_{n+1} \to \Sigma_n A_0$ with $A_0, A_{n+1} \in \sA$ there is an $(n+2)$-angle
         \includepdf{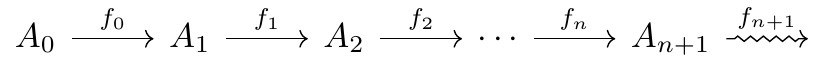}
         with terms $A_1, \dots, A_n \in \sA$.
      \end{enumerate}
   \end{definition}
   Recall that in pre-$(n+2)$-angulated categories as defined in \cite[definition 2.1(a)]{geiss_n-angulated_2013} the class $\pentagon$ of $(n+2)$-angles is closed under direct summands.
   However, since two isomorphic $n$-$\Sigma_n$-sequences are (trivial) direct summands of each other, this implies that $\pentagon$ is closed under isomorphisms inside the class of $n$-$\Sigma_n$-sequences.
   Similarly, any additive subcategory $\sA \subset \sF$ is closed under isomorphisms in $\sF$.

   We will often use this and the following trivial lemma to rename objects:
   \begin{lemma}[Replacement lemma] \label{lemma:replacement_lemma}
      Suppose given an $(n+2)$-angle $X$ of the shape
      \includepdf{tikz/n-plus-2-angle-generic.pdf}
      and a commutative diagram
      \includepdf{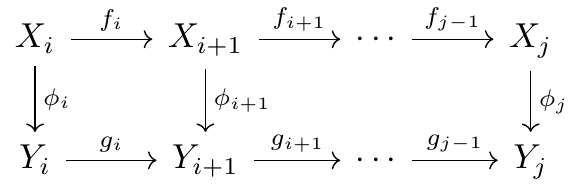}
      with $0 \leq i \leq j \leq n+1$ and $\phi_i,\dots,\phi_j$ isomorphisms.
      Then the $(n+2)$-angle $Y$ given by
      \includepdf{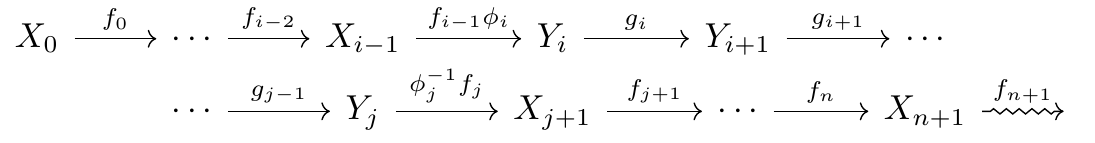}
      is isomorphic to $X$ via $\phi = (\id_{X_0}, \dots, \id_{X_{i-1}}, \phi_i, \dots, \phi_j, \id_{X_{j+1}}, \dots, \id_{X_{n+1}}) \colon X \to Y$.
   \end{lemma}
   \begin{proof}
      The proof is trivial, since the class of $(n+2)$-angles is closed under isomorphisms and the given $\phi$ is obviously an isomorphism of $n$-$\Sigma_n$-sequences.
   \end{proof}
  
   Apart from pre-$(n+2)$-angulated categories, we will also have to deal with $n$-exact categories.
   Recall the following from \cite[definition 2.4]{jasso_n-abelian_2016}:
   An \emph{$n$-exact sequence} in an additive cateogry $\sA$ is a sequence of morphisms 
   \includepdf{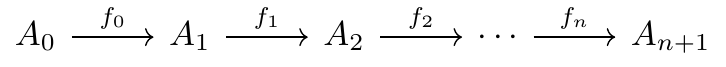}
   such that $f_0$ is a monomorphism, $f_n$ is a epimorphism, $f_i$ is a weak cokernel of $f_{i-1}$ for $i = 1,\dots,n$ and $f_i$ is a weak kernel of $f_{i+1}$ for $i = 0, \dots, n-1$.
   Notice, here $f_0$ being monic is equivalent to $f_0$ being a kernel of $f_1$ and $f_n$ being epic is equivalent to $f_n$ being a cokernel of $f_{n-1}$.
   Also recall from \cite[definition 4.1]{jasso_n-abelian_2016} that a morphism 
   \includepdf{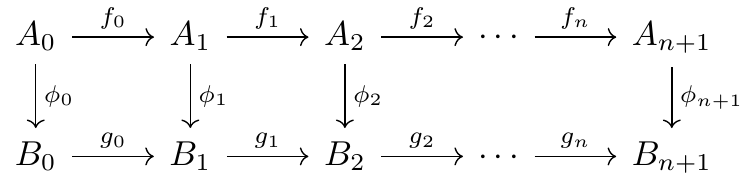}
   from an $n$-exact sequences $E$ to an $n$-exact sequence $F$ is called \emph{weak isomorphism} if $\phi_i$ and $\phi_{i+1}$ for an $i = 0, \dots, n$ or $\phi_0$ and $\phi_{n+1}$ are isomorphisms.
   The sequences $E$ and $F$ are then called \emph{weakly isomorphic}.
   If $A_0 = B_0$ and $A_{n+1} = B_{n+1}$, as well as $\phi_0 = \id_{A_0}$ and $\phi_{n+1} = \id_{A_{n+1}}$ then $E$ and $F$ are called \emph{equivalent} and the morphism between $E$ and $F$ is then called \emph{equivalence of $n$-exact sequences} according to \cite[definition 2.9]{jasso_n-abelian_2016}.

   Recall, by \cite[proposition 2.5(a)]{geiss_n-angulated_2013}, all $(n+2)$-angles are exact, that is applying $\Hom_\sF(X,-)$ or $\Hom_\sF(-,X)$ to an $(n+2)$-angle yields an (infinitely extended) long exact sequence.
   In particular, we will use the following and its dual regularly: 
   Given an $(n+2)$-angle
   \includepdf{tikz/n-plus-2-angle-generic.pdf}
   and a morphism $f \colon X' \to X_i$ for some $i \in \{0,\dots,n+1\}$.
   Then $f$ factors through $f_{i-1}$, where we use the convention $f_{-1} = \Sigma_n^{-1} f_{n+1}$, iff $f f_i = 0$.
   This also shows that $f_{i-1}$ is a weak kernel of $f_{i}$ for $i \in \{0, \dots, n+1\}$.
   Hence, the following trivial lemma and its dual are not only useful in the context of $n$-exact sequences, but $(n+2)$-angulated categories as well:

   \begin{lemma} \label{lemma:weak-cokernels}
      Suppose we are given a commutative diagram as in \cref{figure:weak-cokernels}
      \begin{figure}[h]
      \includepdf{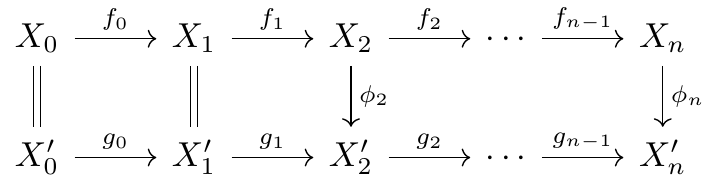}
      \caption{Morphisms between weak cokernels.}
      \label{figure:weak-cokernels}
      \end{figure}
      in an arbitrary additive category $\sA$.
      Suppose 
      \begin{enumerate}
         \item the morphism $f_i$ is a weak cokernel of $f_{i-1}$ for all $i = 1, \dots, n-1$,
         \item the morphism $g_i$ is a weak cokernel of $g_{i-1}$ for all $i = 1, \dots, n-1$ and
         \item the morphisms $f_2, \dots, f_{n-1}$ and $g_2, \dots, g_{n-1}$ are in the radical.
      \end{enumerate}
      Then $\phi_{2}, \dots, \phi_{n-1}$ are isomorphisms.
   \end{lemma}

   Notice, the construction of the inverses of $\phi_2, \dots, \phi_{n-1}$ is similar to the construction of the homotopy inverses in \cite[proposition 2.7]{jasso_n-abelian_2016} and to \cite[lemma 2.1]{jasso_n-abelian_2016}.
   However, we cannot directly deduce \cref{lemma:weak-cokernels} from the statements in \cite{jasso_n-abelian_2016}, so we give a proof for the convenience of the reader:

   \begin{proof}
      We show the lemma under the additional assumption that $X_i = X'_i$ for $i = 0, \dots, n$ and $f_i = g_i$ for $i = 0, \dots, n-1$ first:
      We obtain a commutative diagram 
      \includepdf{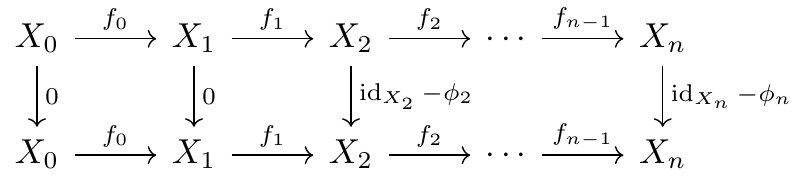}
      by subtracting the identity from each vertical morphism.
      Then $\id_{X_2}-\phi_2$ factors through $f_2$, using $f_1(\id_{X_2}-\phi_2) = 0$ and $f_2$ being a weak cokernel of $f_1$.
      Say $\id_{X_2} - \phi_2 = f_2 h_3$ for some $h_3 \colon X_3 \to X_2$.
      In the same manner we can construct morphisms $h_i \colon X_i \to X_{i-1}$ for $i = 4, \dots, n$ with $\id_{X_{i-1}} - \phi_{i-1} = f_{i-1} h_{i} + h_{i-1} f_{i-2}$ inductively: 
      Suppose all $h_k$ have been constructed for $k = 3, \dots, i-1$ and $4 \leq i \leq n$.
      We have
      \begin{equation} f_{i-2}(\id_{X_{i-1}} - \phi_{i-1} - h_{i-1} f_{i-2}) = (\id_{X_{i-2}} - \phi_{i-2})f_{i-2} - f_{i-2} h_{i-1} f_{i-2} \label{equation:homotopy-lift} \end{equation}
      by commutativity of the above diagram.
      By assumption and because $f_{i-3} f_{i-2} = 0$ holds, or for $i = 4$ because $\id_{X_{i-2}} - \phi_{i-2} = f_2 h_3$ holds, the right hand side of \cref{equation:homotopy-lift} vanishes.
      Since $f_{i-1}$ is a weak cokernel of $f_{i-2}$ the left side of \cref{equation:homotopy-lift} vanishing tells us that $h_{i}$ exists as desired, which completes the induction.
      Now, since $\id_{X_2} - \phi_2 = f_2 h_3$ and $\id_{X_i} - \phi_i = f_i h_{i+1} + h_i f_{i-1}$ for $i = 3, \dots, n-1$ are radical morphisms, this shows that $\phi_i$ is an isomorphism for $i = 2, \dots, n-1$.
      Hence, we have shown the lemma under our additional assumption $X_i = X'_i$ and $f_i = g_i$ for $i = 2, \dots, n-1$.
      
      Now to the general case:
      Using that $g_i$ is a weak cokernel of $g_{i-1}$ for all $i = 1, \dots, n-1$ it is easy to construct the dashed morphisms $\{\phi'_{k} \colon X'_k \to X_k \}_{k=2}^n$ making the diagram in \cref{figure:weak-cokernels-2}
      \begin{figure}[h]
         \includepdf{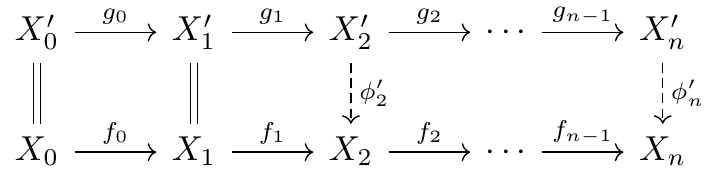}
         \caption{Construction of morphisms in the opposite direction as in \cref{figure:weak-cokernels}.}
         \label{figure:weak-cokernels-2}
      \end{figure}
      commutative. 
      By applying the already proven version of the lemma to the composite of the vertical morphisms of \cref{figure:weak-cokernels} and \cref{figure:weak-cokernels-2} we obtain that $\phi_i \phi'_i$  and $\phi'_i \phi_i$ are isomorphisms for $i = 2, \dots, n-1$ and hence the general version of the lemma follows. 
   \end{proof}

   \section{Core lemmas}
   \subsection{The completion lemma}
   
   We start with the observation that the axiom (F3) from \cite[definition 2.1]{geiss_n-angulated_2013} implies a stronger version of itself.
   Notice, the case $i = 1$ in the following lemma is precisely the axiom (F3):

   \begin{lemma}[Completion lemma] \label{lemma:completion-lemma}
      Suppose given $(n+2)$-angles $X$ and $Y$ as well as $i+1$ morphisms $\{\phi_k \colon X_k \to Y_k\}_{k=0}^i$ which make the diagram
      \includepdf{tikz/N3-asterik-axiom.pdf}
      commutative. 
      If only one map $\phi_0 \colon X_0 \to Y_0$ is given, assume further $(\Sigma_n^{-1} f_{n+1}) \phi_0 g_0 =0$.
      Then there are morphisms $\{\phi_k \colon X_k \to Y_k\}_{k=i+1}^{n+1}$ such that $\phi = (\phi_0,\dots,\phi_{n+1}) \colon X \to Y$ is a morphism of $(n+2)$-angles.
   \end{lemma}
   \begin{proof}
      Let $i+1$ be the number of maps given.
      Notice, for $i = n+1$ there is nothing to show, as $\phi=(\phi_0, \dots, \phi_{n+1})$ is then already a morphism of $(n+2)$-angles.
      Further, for $i = 1$ the lemma is precisely the axiom (F3) and therefore holds.
      For $i = 0$, by the additional assumption and because $(n+2)$-angles are exact, there is a map $\phi_1 \colon X_1 \to Y_1$ with $\phi_0 g_0 = f_0 \phi_1$.
      But then $\phi_0$ and $\phi_1$ satisfy the requirement of the lemma for $i = 1$ and in this case the completion has already been established.
      We show the remaining cases by induction on $i$.
      As we already established the start of the induction, suppose the lemma is true for some fixed $1 \leq i \leq n-1$.
      We show that we can complete any $i+2$ morphisms $\{\phi_k \colon X_k \to Y_k\}_{k=0}^{i+1}$ to a morphism of $(n+2)$-angles, as well:

      By the induction hypothesis, we can apply the lemma to the morphisms $\{\phi_k\}_{k=0}^i$ and obtain a morphism $\phi' \colon X \to Y$ of $(n+2)$-angles satisfying $\phi_k = \phi'_k$ for $0 \leq k \leq i$.
      We obtain a commutative diagram
      \includepdf{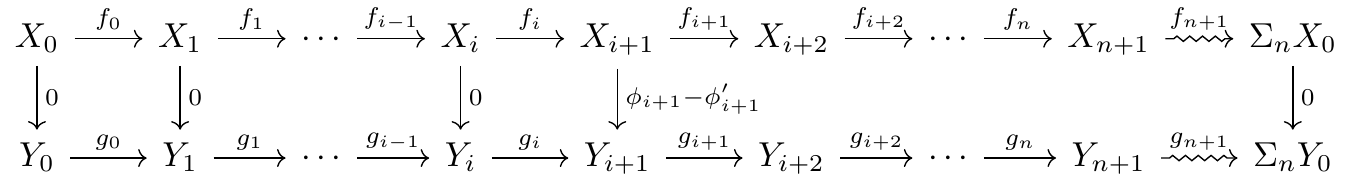}
      by subtraction of $\phi'_k$ from $\phi_k$ for $0 \leq k \leq i+1$. 
      Therefore, $f_i(\phi_{i+1} - \phi'_{i+1}) = 0$ and because $(n+2)$-angles are exact, this shows that there is a morphism $h_{i+2} \colon X_{i+2} \to Y_{i+1}$ with $\phi_{i+1} - \phi'_{i+1} = f_{i+1} h_{i+2}$.
      Hence, $f_{i+1} h_{i+2} g_{i+1} = (\phi_{i-1} -\phi'_{i+1}) g_{i+1}$ and $h_{i+2} g_{i+1} g_{i+2} = 0$, so the diagram
      \includepdf{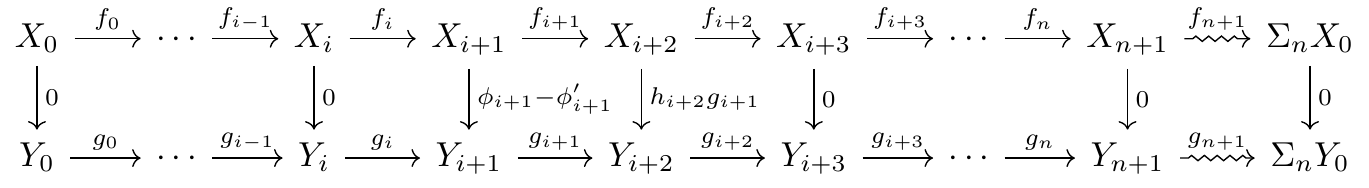}
      is commutative.
      This yields a morphism $\psi \colon X \to Y$ of $(n+2)$-angles defined by $\psi_k = 0$ for $k \notin \{i+1, i+2\}$ and $\psi_{i+1} = \phi_{i+1} - \phi'_{i+1}$ as well as $\psi_{i+2} = h_{i+2} g_{i+1}$.
      Notice, since $\phi'_k = \phi_k$ for $0 \leq k \leq i$ this means that the diagram
      \includepdf{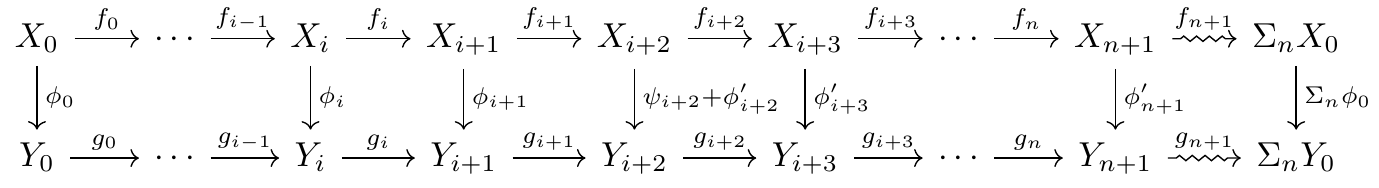}
      which is obtained by the summation of $\psi$ and $\phi'$, is commutative.
      Therefore $\psi + \phi'$ is the desired completion of the given $\{\phi_{k}\}_{k=0}^{i+1}$ to a morphism of $(n+2)$-angles and the lemma hence follows by induction.
   \end{proof}

   One immediate consequence of the completion \cref{lemma:completion-lemma} is the following, which is a generalization of \cite[corollaire 1.1.10(ii)]{beilinson_faisceau_pervers_82} to the context of $(n+2)$-angulated categories.
   It will later establish the isomorphism $\YExt_{(\sA, \sE_{\sA})}^n(A_{n+1}, A_0) \to \Hom_{\sF}(A_{n+1}, \Sigma_n A_0)$ for $A_0, A_{n+1} \in \sA$, similarly to \cite[proposition 2.5(ii)]{joergensen_simple_minded_2020}.

   \begin{lemma} \label{lemma:connecting-morphism-is-unique}
      Suppose given a commutative diagram
      \includepdf{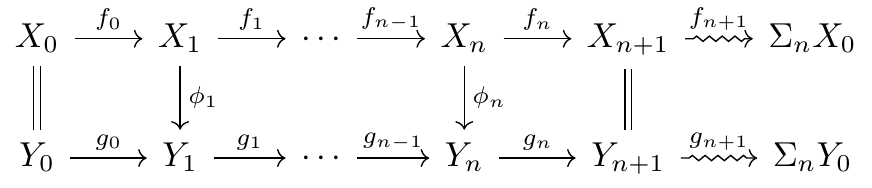}
      with $(n+2)$-angles as rows and $\Hom(\Sigma_n X_0, Y_{n+1}) = 0$.
      Then $f_{n+1} = g_{n+1}$.
   \end{lemma}
   \begin{proof}
   The undashed maps of
   \includepdf{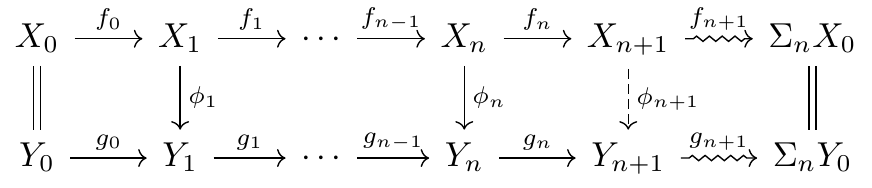}
   clearly form a commutative diagram.
   We claim there is at most one $\phi_{n+1} \colon X_{n+1} \to Y_{n+1}$ which makes the penultimate square commute, i.e.\ satisfies $f_n \phi_{n+1} = \phi_n g_n$:
   Suppose we are given two maps $\phi_{n+1} \colon X_{n+1} \to Y_{n+1}$ and $\phi'_{n+1} \colon X_{n+1} \to Y_{n+1}$, both having this property.
   Then we obtain $f_n (\phi_{n+1} - \phi'_{n+1}) = 0$ and hence, using that the upper row is an exact $n$-$\Sigma_n$-sequence, there is a morphism $h \colon \Sigma_n X_0 \to Y_{n+1}$ with $f_{n+1} h = \phi_{n+1} - \phi'_{n+1}$.
   However, $\Hom(\Sigma_n X_0, Y_{n+1}) = 0$ by assumption, hence $h = 0$ and $\phi_{n+1} = \phi'_{n+1}$.
   This shows that there is at most one $\phi_{n+1}$ making the penultimate square commute.

   By the completion lemma, there is a choice of $\phi_{n+1}$ which makes the whole diagram commutative.
   But by the previous claim, this choice is unique and, by requirement of the lemma, its the identity on $X_{n+1} = Y_{n+1}$.
   This means that 
   \includepdf{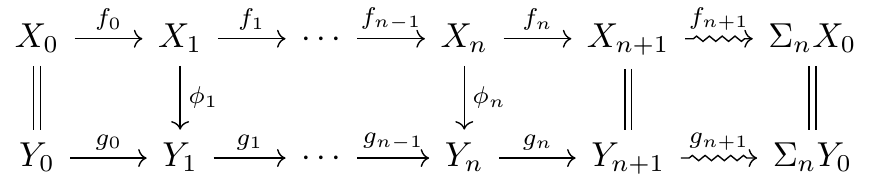}
   is a commutative diagram, which shows $f_{n+1} = g_{n+1}$.
   \end{proof}
  
   \subsection{Obtaining information from (N4*)} \label{subsection:about_N4}
   Throughout \cref{subsection:about_N4} we assume that $(\sF,\Sigma_n,\pentagon)$ is an $(n+2)$-angulated Krull-Schmidt category.
   Recall from \cite[theorem 4.4]{bergh_axioms_2013} that given any commutative diagram as in \cref{figure:N4-asterik-axiom-1}
   \begin{figure}[h]
      \includepdf{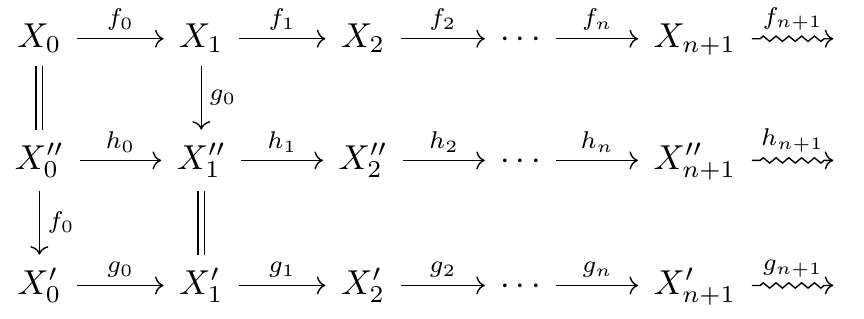}
      \caption{A factorization $h_0 = f_0 g_0$ and the three $(n+2)$-angles arising from the morphisms involved. Notice that $X_0' = X_1$.}
      \label{figure:N4-asterik-axiom-1}
   \end{figure}
   with all three rows being $(n+2)$-angles, we can find the dashed morphisms of \cref{figure:N4-asterik-axiom-2}
   \begin{figure}[h]
      \includepdf{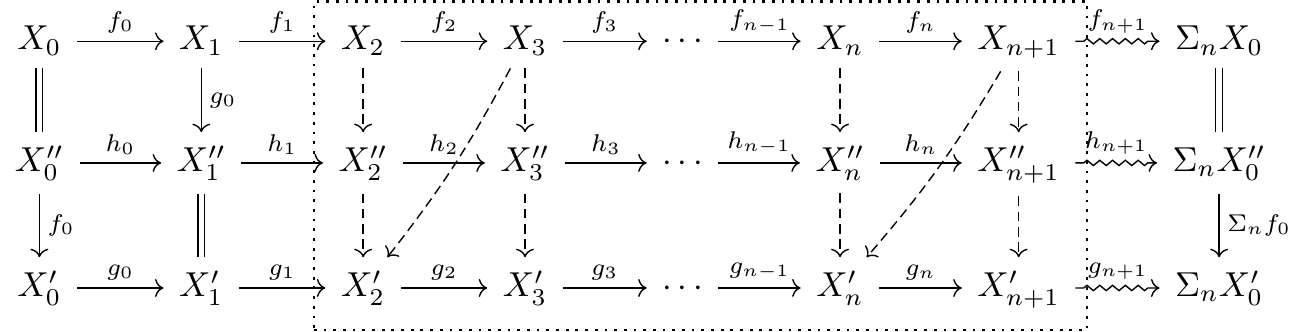}
      \caption{Morphisms arising from \cref{figure:N4-asterik-axiom-1} via (N4*).}
      \label{figure:N4-asterik-axiom-2}
   \end{figure}
   such that each upright square commutes and so that the (Mayer-Vietoris-like) totalisation of the complex enclosed by the dashed rectangle in \cref{figure:N4-asterik-axiom-2}, shown in \cref{figure:N4-asterik-axiom-totalisation} 
   \begin{figure}[h]
      \includepdf{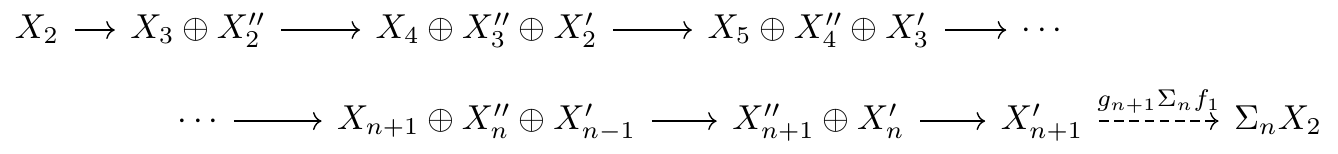}
      \caption{Totalisation obtained from \cref{figure:N4-asterik-axiom-2}.}
      \label{figure:N4-asterik-axiom-totalisation}
   \end{figure}
   is an $(n+2)$-angle.
   In this section we want to tackle the question, what information about $X_2'',\dots, X_{n+1}''$ can be obtained from this $(n+2)$-angle?
   Notice, for $n > 1$ the completion of a morphism to an $(n+2)$-angle is not unique, as we can always add a trivial $(n+2)$-angle to an existing completion to obtain another one.
   To circumvent this problem we reintroduce the following from \cite[lemma 5.18]{oppermann_higher-dimensional_2012} and \cite[lemma 3.14]{Fedele_2019}, which adds uniqueness up to isomorphism to $(n+2)$-angles:
   \begin{definition} \label{definition:minimal-n+2-angle}
      Suppose given an $(n+2)$-angle $X$ of the shape
      \includepdfdot{tikz/n-plus-2-angle-generic.pdf}
      Then $X$ is called 
      \begin{enumerate}
         \item a \emph{minimal completion of $f_0$} if $f_i \in \rad(X_i, X_{i+1})$ for $i = 2,\dots,n$ and
         \item a \emph{minimal $(n+2)$-angle for $f_i$} if its rotation  \label{enum:minimal-n+2-angle}
         \includepdf{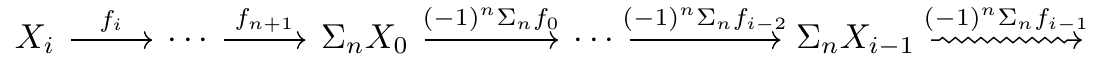}
         is a minimal completion of $f_i$.
      \end{enumerate}
   \end{definition}
   Notice, since the rotation of an $(n+2)$-angle is an $(n+2)$-angle again it does not really matter in what position $i$ the morphism $f_i$ is in (\ref{enum:minimal-n+2-angle}).
   For sake of convenience and less heavy notation, we will state and prove the lemmas regarding minimal $(n+2)$-angles for minimal completions only and leave the rotated versions for minimal $(n+2)$-angles as an easy exercise for the reader.

   Most of the properties of minimal $(n+2)$-angles can be shown without the Krull-Schmidt property of $\sF$, but we need it for the existence of minimal $(n+2)$-angles and hence to use the main result of this section:
   A lemma, which enables us to compare the objects of the $(n+2)$-angle arising as in \cref{figure:N4-asterik-axiom-totalisation} and the objects of the minimal $(n+2)$-angle arising from the morphism $g_{n+1} \Sigma_n f_1$ in position $n+1$ of this $(n+2)$-angle.
   \begin{lemma} \label{lemma:comparison-lemma}
      Suppose $n \geq 3$ and we are given a commutative diagram as in \cref{figure:N4-asterik-axiom-1}, where the second row is additionally a minimal $(n+2)$-angle for $h_0$.
      Further let
      \includepdf{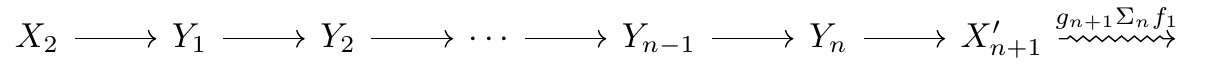}
      be a minimal $(n+2)$-angle for $g_{n+1}\Sigma_n f_1$.
      Then there are objects $\{Z_k\}_{k=1}^{n-1}$ satisfying
      \begin{enumerate}
         \item $Z_1 \in \add(X_3 \oplus X_4 \oplus X'_2)$, \label{enum:begin-Z-add-property}
         \item $Z_k \in \add(X_{k+2} \oplus X'_{k} \oplus X_{k+3} \oplus X'_{k+1})$ for $2 \leq k \leq n-2$ and
         \item $Z_{n-1} \in \add(X_{n+1} \oplus X'_{n-1} \oplus X'_{n})$, \label{enum:end-Z-add-property}
      \end{enumerate}
      expressing the difference between the objects of \cref{figure:N4-asterik-axiom-totalisation} and the minimal $(n+2)$-angle of $g_{n+1} \Sigma_n f_1$ depicted above, in the sense that
      \begin{enumerate}[resume]
         \item $X_3 \oplus X_2'' = Y_1 \oplus Z_1$, \label{enum:begin-sum-property}
         \item $X_{k+2} \oplus X_{k+1}'' \oplus X'_k = Z_{k-1} \oplus Y_k \oplus Z_k$ for $2 \leq k \leq n-1$ and
         \item $X''_{n+1} \oplus X_n' = Z_{n-1} \oplus Y_n $ \label{enum:end-sum-property}
      \end{enumerate}
      hold.
   \end{lemma}
   This lemma seems very technical, though it is very useful, too.
   For example it can be used to show that $X''_2,\dots,X''_{n+1}$ lie in an additive $n$-extension closed subcategory of $\sF$ if the upper and lower row of \cref{figure:N4-asterik-axiom-1} do so.

   Notice the case $n=1$ of \cref{lemma:comparison-lemma} is uninteresting, since triangles in a triangulated category are unique up to isomorphism.
   A similar form of \cref{lemma:comparison-lemma} holds for $n=2$. Here (\ref{enum:begin-Z-add-property})-(\ref{enum:end-Z-add-property}) need to be replaced by $Z_1 \in \add(X_3 \oplus X_2')$.

   Our plan of proving this is as follows:
   We show that the $(n+2)$-angles of \cref{figure:N4-asterik-axiom-totalisation} and \cref{lemma:comparison-lemma} differ by a direct sum of trivial complexes.
   The objects of these trivial complexes are our candidates for $\{Z_k\}_{k=1}^{n-1}$.
   We then use the matrix lemma (see \cref{lemma:matrix-lemma}) and $h_2,\dots,h_{n}$ being in the radical to obtain (\ref{enum:begin-Z-add-property})--(\ref{enum:end-Z-add-property}).
   We begin with the first part of this plan:
   
   \begin{lemma} \label{lemma:uniqueness-of-minimal-n+2-angles}
      Let $f_0 \colon X_0 \to X_{1}$ be a morphism.
      Suppose $X$ is a minimal completion 
      \includepdf{tikz/n-plus-2-angle-generic.pdf}
      of $f_0$. 
      Then $X$ is a direct summand of any $(n+2)$-angle containing $f_0$ in position $0$.
      Moreover, the minimal completion of $f_0$ is unique up to isomorphism.
   \end{lemma}
   \begin{proof}
      In a triangulated category the lemma is clearly true, as triangles are, up to isomorphism, uniquely determined by one morphism.
      Hence, we can assume that $n \geq 2$.

      Suppose $Y$ is an arbitrary $(n+2)$-angle
      \includepdf{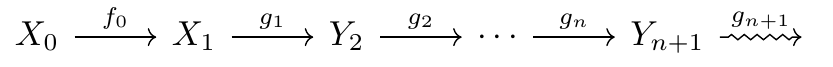}
      containing $f_0$ in the same position as $X$.
      Using the axiom \cite[defintion 2.1(F3)]{geiss_n-angulated_2013} we find the dashed morphisms $\phi_2, \dots, \phi_{n+1}$ and $\phi'_2, \dots, \phi'_{n+1}$ making
      \includepdf{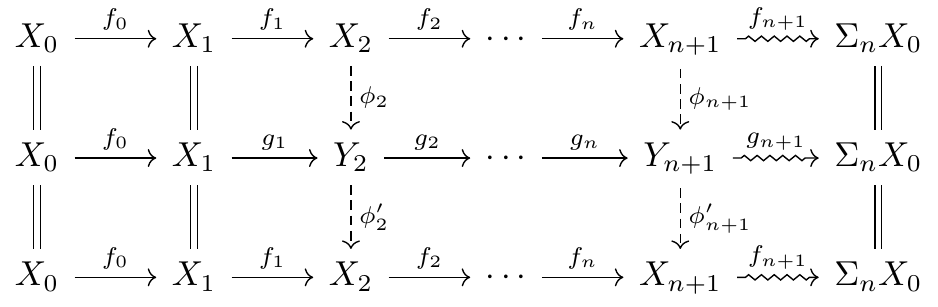}
      a commutative diagram.
      For short, that is to say, we have morphisms of $(n+2)$-angles $\phi = (\id_{X_0}, \id _{X_1}, \phi_2, \dots, \phi_{n+1}) \colon X \to Y$ and $\phi' = (\id_{X_0}, \id _{X_1}, \phi'_2, \dots, \phi'_{n+1}) \colon Y \to X$.
      Now by \cref{lemma:weak-cokernels} we know that $\phi_2 \phi'_2, \dots, \phi_{n} \phi'_n$ are isomorphisms.
      By the dual of \cref{lemma:weak-cokernels}, we obtain that $\phi_3\phi'_3, \dots, \phi_{n+1} \phi'_{n+1}$ are isomorphisms, hence $\phi \phi'$ is an isomorphism and $X$ is a direct summand of $Y$.

      Now if $Y$ is additionally a minimal $(n+2)$-angle for $f_0$ then the same argument using \cref{lemma:weak-cokernels} and its dual shows that $\phi$ has to be an isomorphism already.
      This shows that any two minimal $(n+2)$-angles of $f_0$ are isomorphic.
   \end{proof}

   \begin{lemma}\label{lemma:minimal-n-plus-2-angles-exist}
      Each morphism $f_0 \colon X_0 \to X_1$ has a minimal completion.
      Moreover, any completion of $f_0$ to an $(n+2)$-angle is the direct sum of the minimal completion of $f_0$ and trivial $(n+2)$-angles $\triv_2 (Z_2), \dots, \triv_{n}(Z_{n})$ for some $Z_2, \dots, Z_{n} \in \sF$.
   \end{lemma}
   \begin{proof}
      As each morphism in $\sF$ has a completion to an $(n+2)$-angle the first part of the lemma follows from the second part.

      We show the second part:
      Suppose we are given an $(n+2)$-angle $X$ of the form
      \includepdf{tikz/n-plus-2-angle-generic.pdf}
      containing $f_0$.
      Suppose $f_i$ is not in the radical for some $i = 2,\dots, n$.
      By definition there is a $0 \neq Z$ and $g \colon Z \to X_i$ and $h \colon X_{i+1} \to Z$ such that $gf_ih$ is an isomorphism.
      Then commutativity of
      \includepdf{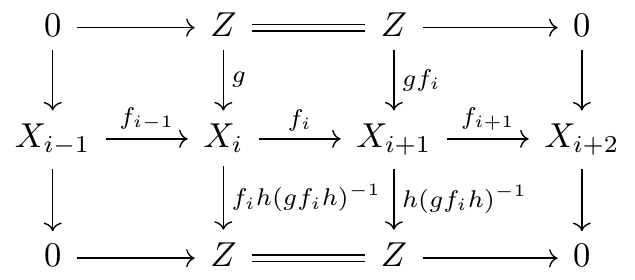}
      shows that the trivial $(n+2)$-angle $\triv_i(Z)$ is a direct summand of $X$.
      Choose a complement $X'$ for $\triv_i(Z)$ in $X$, that is $X = X' \oplus \triv_i(Z)$ as $(n+2)$-angles.
      Since $i \neq 0, 1, n+1$ we know that $X'$ has $f_0$ in position $0$.
      If $X'$ is not a minimal completion of $f_0$ we can repeat this procedure.
      As any object $X_0, \dots, X_{n+1}$ has the descending chain condition on direct summands, the $(n+2)$-angle $X$ has the descending chain condition on direct summands.
      This means, after finitely many steps this procedure must stop and hence $X$ is the direct sum of a minimal completion of $f_0$ and trivial $(n+2)$-angles.
      Finally, noticing $\triv_i(Z) \oplus \triv_i(Z') = \triv_i(Z\oplus Z')$ for $i \in \{2,\dots,n\}$ and $Z,Z' \in \sF$ and thus collecting all the trivial $(n+2)$-angles of the same shape in one summand yields the lemma.
   \end{proof}
   
   Notice, by \cref{lemma:minimal-n-plus-2-angles-exist} minimal completions exist and by \cref{lemma:uniqueness-of-minimal-n+2-angles} they are unique up to isomorphism. 
   We can therefore speak of \emph{the} minimal completion of a morphism $f_0$.
   The same holds for minimal $(n+2)$-angles. 

   For the second part of our plan we show the following easy lemma, which holds in any additive Krull-Schmidt category:

   \begin{lemma}[Matrix lemma] \label{lemma:matrix-lemma}
      Let $f \colon X_1 \oplus X_2 \to Y_1 \oplus Y_2$ be so that $\iota_1f\pi'_1 \in \rad(X_1, Y_1)$, where $\iota_1 \colon X_1 \to X_1 \oplus X_2$ is the canonical inclusion and $\pi'_1 \colon Y_1 \oplus Y_2 \to Y_1$ is the canonical projection. 
      Then all objects $Z$ which have $g \colon Z \to X_1 \oplus X_2$ and $h \colon Y_1 \oplus Y_2 \to Z$ such that $gfh \in \Aut Z$ satisfy $Z \in \add(X_2 \oplus Y_2)$.
   \end{lemma}
   \begin{proof}
      First assume that $Z$ is indecomposable, with $g$ and $h$ as in the lemma.
      For $i=1,2$ let the morphisms $\iota_i \colon X_i \to X_1 \oplus X_2$ and $\iota'_i \colon Y_i \to Y_1 \oplus Y_2$ be the canonical inclusions and $\pi_i \colon X_1 \oplus X_2 \to X_i$ and $\pi_i' \colon Y_1 \oplus Y_2 \to Y_i$ be the canonical projections with respect to the given direct sum decomposition.
      We know that $\rad(Z,Z) = \rad(\End Z)$ is local, since $Z$ is indecomposable.
      Now we have 
      \[ \sum_{(i,j) \in \{1,2\}^2}g\pi_{i}\iota_{i}f\pi'_{j}\iota'_{j}h = gfh \in \Aut(Z) \text{,}\]
      so at least one of the four summands on the left is an isomorphism.
      However, since the radical is an ideal, we have $g\pi_{1}\iota_{1} f \pi'_{1} \iota'_{1} h \in \rad( \End Z)$, so one of the other three summands needs to be an isomorphism.
      But then it must be a summand with $(i,j) \neq (1,1)$, so $g\pi_{2}$ is split mono or $\iota_{2}' h$ is split epi.
      This shows the lemma for $Z$ indecomposable.

      Now let $Z$ be arbitrary, with $g$ and $h$ as in the lemma. 
      Let $Z'$ be an indecomposable direct summand of $Z$, that is there are $\iota \colon Z' \to Z$ and $\pi \colon Z \to Z'$ with $\iota \pi = \id_{Z'}$.
      Then $\id_{Z'} = g' f h'$ with $g' = \iota g$ and $h' = h(gfh)^{-1}\pi$.
      Hence $Z' \in \add(X_2 \oplus Y_2)$ by the first part.
      But $Z$ decomposes, up to isomorphism and reordering, uniquely into finitely many indecomposable objects.
      As all indecomposable direct summands of $Z$ belong to $\add(X_2 \oplus Y_2)$ we finally obtain $Z \in \add(X_2 \oplus Y_2)$.
   \end{proof}
  
   We are now able to prove the main result of this section.

   \begin{proof}[Proof of \cref{lemma:comparison-lemma}] \label{proof:comparison-lemma}
   We use the same notation as in \cref{lemma:comparison-lemma}.
   We apply the axiom (N4*) to \cref{figure:N4-asterik-axiom-1}.
   We obtain the $(n+2)$-angle in \cref{figure:N4-asterik-axiom-totalisation} and \cref{figure:totalization-labled}
   \begin{figure}[ht]
      \includepdf{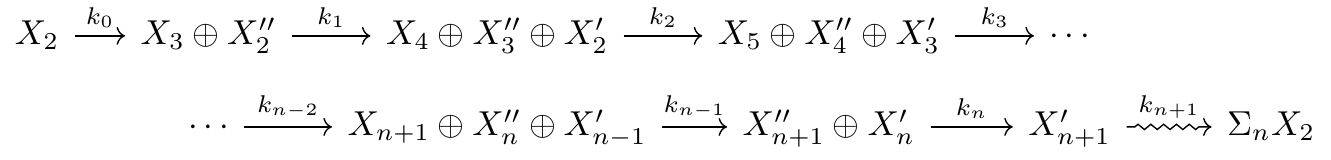}
      \caption{Labeled version of the totalisation in \cref{figure:N4-asterik-axiom-totalisation}.}
      \label{figure:totalization-labled}
   \end{figure}
   where the morphisms $k_0, \dots, k_{n+1}$ can be written in the form
   \begin{align*}
      k_0 &= \left[\begin{smallmatrix} f_2 \\ \phi_2  \end{smallmatrix}\right] \\
      k_1 &= \left[\begin{smallmatrix} -f_3 & 0 \\ \phi_3 & -h_2 \\ \psi_3 & \phi'_2 \end{smallmatrix}\right] \\
      k_i &= \left[\begin{smallmatrix} -f_{i+2} & 0 & 0 \\ \phi_{i+2} & -h_{i+1} & 0 \\ \psi_{i+2} & \phi'_{i+1} & g_i  \end{smallmatrix}\right] \tag{$2 \leq i \leq n-2$}\\
      k_{n-1} &= \left[\begin{smallmatrix} \phi_{n+1} & -h_{n} & 0 \\ \psi_{n+1} & \phi'_{n} & g_{n-1} \end{smallmatrix}\right] \\
      k_n &= \left[\begin{smallmatrix} \phi'_{n+1} & g_{n} \end{smallmatrix}\right] \\
      k_{n+1} &= g_{n+1} \Sigma_n f_1
   \end{align*}
   for some morphisms $\phi_k \colon X_k \to X''_k$ and $\phi'_k \colon X''_k \to X'_k$ for $2 \leq k \leq n+1$ as well as morphisms $\psi_{k} \colon X_k \to X'_{k-1}$ for $3 \leq k \leq n+1$.
   By \cref{lemma:minimal-n-plus-2-angles-exist} we know that the $(n+2)$-angle shown in \cref{figure:totalization-labled} is the direct sum of the minimal $(n+2)$-angle of $k_{n+1}$ and trivial $(n+2)$-angles $\triv_{1}(Z_1), \dots, \triv_{n-1}(Z_{n-1})$.
   This choice of $Z_1, \dots, Z_{n-1}$ establishes (\ref{enum:begin-sum-property})-(\ref{enum:end-sum-property}) of \cref{lemma:comparison-lemma}.

   It remains to show that (\ref{enum:begin-Z-add-property})-(\ref{enum:end-Z-add-property}) of \cref{lemma:comparison-lemma} hold for this choice of $Z_1,\dots,Z_{n-1}$.
   We only show (\ref{enum:begin-Z-add-property}), the other assertions are analogous.
   Since the trivial $(n+2)$-angle $\triv_1(Z_1)$ is a direct summand of the $(n+2)$-angle in \cref{figure:totalization-labled} we know that $\id_{Z_1}$ factors through $k_1$.
   Since $-h_2$ is in the radical by assumption we must have $Z_1 \in \add(X_3 \oplus (X_4 \oplus X'_2))$ by the matrix \cref{lemma:matrix-lemma}.
   \end{proof}
   
   \section{Main results} \label{section:main_results} 
   Throughout \cref{section:main_results} we assume that $(\sF,\Sigma_n,\pentagon)$ is an $(n+2)$-angulated Krull-Schmidt category.
   Further, we will assume that $\sA$ is an additive, $n$-extension closed subcategory of $\sF$.
   We define an $n$-exact structure on $\sA$:

   \begin{definition}
      For a subcategory $\sA \subset \sF$ an \emph{$\sA$-conflation} is a complex
      \includepdf{tikz/short-n-plus-2-angle.pdf}
      with $A_0, \dots, A_{n+1} \in \sA$ for which there is a morphism $f_{n+1} \colon A_{n+1} \to \Sigma_n A_0$ such that
      \includepdf{tikz/n-plus-2-angle.pdf}
      is an $(n+2)$-angle in $(\sF, \Sigma_n, \pentagon)$. 
      The morphisms in position $0$ of $\sA$-conflations are called \emph{$\sA$-inflations}.
      Dually, the morphisms which appear in position $n$ of $\sA$-conflations are called \emph{$\sA$-deflations}.
      We denote by $\sE_{\sA}$ the class of all $\sA$-conflations.
   \end{definition}

   Notice, by the replacement \cref{lemma:replacement_lemma}, it follows immediately that any complex in $\sA$ isomorphic to an $\sA$-conflation is an $\sA$-conflation itself.

   We now want to present our main theorem:

   \begin{theorem} \label{theorem:maintheorem}
      Let $\sA \subset \sF$ be an additive, $n$-extension closed subcategory of $\sF$ with $\Hom_{\sF}(\Sigma_n \sA, \sA) = 0$.
      Then 
      \begin{enumerate}
         \item $(\sA, \sE_{\sA})$ is an $n$-exact category and \label{enum:maintheorem-1}
         \item there is a natural bilinear isomorphism $\Hom_{\sF}(-, \Sigma_n(-)) \to \YExt^n_{(\sA, \sE_{\sA})}(-, -)$ of functors $\sA \times \sA^{\op} \to \Ab$. \label{enum:maintheorem-2}
      \end{enumerate}
   \end{theorem}

   To check part (\ref{enum:maintheorem-1}) of \cref{theorem:maintheorem} we need to verify the axioms of \cite[definition 4.2]{jasso_n-abelian_2016}. 
   For the convenience of the reader this is split up into \cref{lemma:A-conflations-are-n-exact}, \cref{lemma:weak-isomorphism-closure}, \cref{lemma:E0-axiom}, \cref{lemma:E1-axiom} and \cref{lemma:E2-axiom}.
   We will check part (\ref{enum:maintheorem-2}) separately in \cref{lemma:Yext-isomorphism}.

   First we need to verify that $\sE_{\sA}$ really consist of $n$-exact sequences:
   
   \begin{lemma} \label{lemma:A-conflations-are-n-exact}
      All $\sA$-conflations are $n$-exact sequences in $\sA$.
   \end{lemma}
   \begin{proof}
      Suppose we are given an arbitrary $\sA$-conflation 
      \includepdf{tikz/short-n-plus-2-angle.pdf}
      in $\sE_{\sA}$. 
      By definition there is a morphism $f_{n+1} \colon A_{n+1} \to \Sigma_{n} A_0$ such that 
      \includepdf{tikz/n-plus-2-angle.pdf}
      is an $(n+2)$-angle in $(\sF,\Sigma_n,\pentagon)$.
      Applying the functor $\Hom_{\sF}(A,-)$ for $A \in \sA$ to this $(n+2)$-angle and using that $(n+2)$-angles are exact, we obtain an exact sequence
      \includepdf{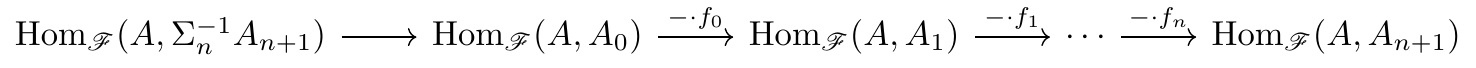}
      where the leftmost term vanishes since $0 = \Hom_{\sF}(\Sigma_n \sA, \sA) \cong \Hom_{\sF}(\sA, \Sigma_n^{-1} \sA)$.
      This statement and its dual statement obtained from applying $\Hom_{\sF}(-, A)$ for $A \in \sA$ to the same $(n+2)$-angle show that any sequence of $\sE_{\sA}$ is indeed $n$-exact.
   \end{proof}
  
   Next, we need to check that an $n$-exact sequence in $\sA$ weakly isomorphic to an $n$-exact sequence in $\sE_{\sA}$ is itself in $\sE_{\sA}$.
   We could prove this directly.
   However, for the sake of readability, we restate a version of \cref{lemma:minimal-n-plus-2-angles-exist} for $n$-exact sequences in $\sA$: 
   
   \begin{lemma} \label{lemma:minimal-conflations-1}
   Suppose we are given an $n$-exact sequence $E$ in $\sA$ of the shape
   \includepdf{tikz/short-n-plus-2-angle.pdf}
   and a fixed $i \in \{0,\dots,n\}$. 
   Then $E$ is the direct sum of an $n$-exact sequence $E'$ of the shape
   \includepdfcomma{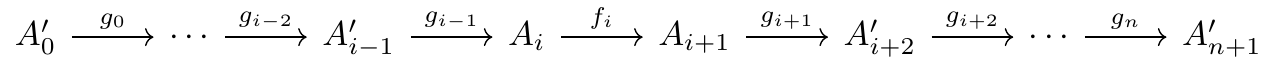}
   where $g_0, \dots, g_{i-2},g_{i+2}, g_n$ are in the radical, and $\sA$-conflations $F_0, \dots,F_{i-2}, F_{i+2}, \dots, F_n$ arising from trivial $(n+2)$-angles $\triv(B_0), \dots, \triv(B_{i-2}), \triv(B_{i+2}), \dots, \triv(B_n)$ with $B_0, \dots, B_{i-2}, B_{i+2}, \dots, B_n \in \sA$.
   Further, if $E'$ is an $\sA$-conflation then so is $E$.
   \end{lemma}
   \begin{proof}
      The $\sA$-conflations $F_0, \dots,F_{i-2}, F_{i+2}, \dots, F_n$ can be constructed in the same way as in \cref{lemma:minimal-n-plus-2-angles-exist}, using that $\sA \subset \sF$ is additive.
      Clearly a direct summand of an $n$-exact sequence in $\sA$ is an $n$-exact sequence in $\sA$, hence $E'$ constructed similarly to \cref{lemma:minimal-n-plus-2-angles-exist} is indeed an $n$-exact sequence in $\sA$.
      Since the direct sum of $\sA$-conflations is again an $\sA$-conflation, as the direct sum of $(n+2)$-angles is again an $(n+2)$-angle, we have that $E$ is an $\sA$-conflation if $E'$ is an $\sA$-conflation.
   \end{proof}
   
   Obviously, \cref{lemma:minimal-conflations-1} also has the following version, where the end terms of the $n$-exact sequences are fixed. 
   The proof is omitted as it is almost the same as of \cref{lemma:minimal-conflations-1}.
   Notice though, that here the $n$-Yoneda-extension class rather than a morphism is fixed.
   
   \begin{lemma} \label{lemma:minimal-conflations-2}
   Suppose we are given an $n$-exact sequence $E$ in $\sA$ of the shape
   \includepdfdot{tikz/short-n-plus-2-angle.pdf}
   Then $E$ is the direct sum of an $n$-exact sequence $E'$ of the shape
   \includepdfcomma{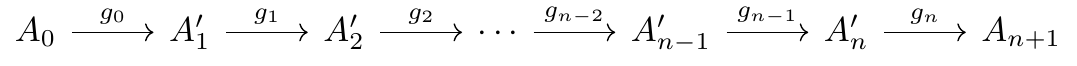}
   where $g_1, \dots, g_{n-1}$ are in the radical, and $\sA$-conflations $F_1, \dots, F_{n-1}$ arising from trivial $(n+2)$-angles $\triv(B_1), \dots, \triv(B_{n-1})$ with $B_1, \dots, B_{n-1} \in \sA$.
   Furthermore, if $E'$ is an $\sA$-conflation then so is $E$.
   \end{lemma}

   \begin{lemma} \label{lemma:weak-isomorphism-closure}
      The class of $\sA$-conflations $\sE_{\sA}$ is closed under weak isomorphisms in the class of $n$-exact sequences in $\sA$.
   \end{lemma}
   \begin{proof}
      We show that for any weak isomorphism $\phi = (\phi_0, \dots, \phi_{n+1})$ of $n$-exact sequences
      \includepdf{tikz/weak-equivalence.pdf}
      in $\sA$ the upper $n$-exact sequence $E$ belongs to $\sE_{\sA}$ if and only if the lower sequence $F$ belongs to $\sE_{\sA}$.
      Therefore, we distinguish the following two cases:

      \emph{Suppose $\phi_i$ and $\phi_{i+1}$ are isomorphisms for some $0 \leq i \leq n$.}
      Then we may assume that $E$ belongs to $\sE_{\sA}$ and show that $F$ then belongs to $\sE_{\sA}$ as the proof with reverse roles of $E$ and $F$ is analogous.
      Further, by the replacement \cref{lemma:replacement_lemma}, we may assume $A_i = B_i$ and $A_{i+1} = B_{i+1}$, as well as $\phi_i = \id_{A_i}$ and $\phi_{i+1} = \id_{A_{i+1}}$.
      Now let $\iota \colon E' \to E$ be the inclusion of the minimal $(n+2)$-angle with $f_i$ in position $i$, which is a direct summand of $E$ by \cref{lemma:minimal-n-plus-2-angles-exist}.
      Further let $\pi \colon F \to F'$ be the projection to the direct summand constructed in \cref{lemma:minimal-conflations-1}.
      By construction $\iota \phi \pi$ is still a weak isomorphism with identities in position $i$ and $i+1$.
      Further, since the class $\pentagon$ of $(n+2)$-angles is closed under direct summands, $E'$ is an $\sA$-conflation and by \cref{lemma:minimal-conflations-1} our sequence $F$ is an $\sA$-conflation if $F'$ is so.
      By replacing $(E,F,\phi)$ by $(E',F',\iota \phi \pi)$, without changing any notation, we can assume that $f_0, f_1, \dots, f_{i-2}, f_{i+2}, \dots, f_{n}$ and $g_0, \dots, g_{i-2},g_{i+2}, \dots, g_n$ are in the radical of $\sF$, hence also in the radical of $\sA$.
      If we can now show that $\phi$ is an isomorphism, then $F$ is an $\sA$-conflation by the replacement \cref{lemma:replacement_lemma}.
      Now notice, because $E$ is $n$-exact, $f_i$ is a weak cokernel of $f_{i-1}$ for $i = 1, \dots, n$ and $0 \colon A_{n+1} \to 0$ is a weak cokernel of the epimorphism $f_n$ in $\sA$.
      Similarly $g_i$ is a weak cokernel of $g_{i-1}$ for $i = 1, \dots, n$ and $0 \colon B_{n+1} \to 0$ is a weak cokernel of $g_{n}$ in $\sA$.
      Hence, \cref{lemma:weak-cokernels} shows that $\phi_{i+2}, \dots, \phi_{n+1}$ are isomorphisms.
      By the dual of \cref{lemma:weak-cokernels} also $\phi_0, \dots, \phi_{i-1}$ are isomorphisms.
      Therefore $\phi$ is an isomorphism and the lemma follows if $\phi_i$ and $\phi_{i+1}$ are isomorphisms for some $0 \leq i \leq n$.

      \emph{Suppose $\phi_0$ and $\phi_{n+1}$ are isomorphisms.}
      Without loss of generality, we can again assume that $E$ belongs to $\sE_{\sA}$ and show that $F$ then belongs to $\sE_{\sA}$.
      Therefore, we have a morphism $f_{n+1} \colon A_{n+1} \to \Sigma_n A_0$ completing $E$ to an $(n+2)$-angle $X$.
      Further, by using the replacement \cref{lemma:replacement_lemma} twice, we may assume that $A_0 = B_0$ and $A_{n+1} = B_{n+1}$ as well as $\phi_0 = \id_{A_0}$ and $\phi_{n+1} = \id_{A_{n+1}}$.
      Similarly to the first part of the proof, using \cref{lemma:minimal-n-plus-2-angles-exist} and \cref{lemma:minimal-conflations-2}, we may assume that the morphisms $f_1, \dots, f_{n-1}$ and $g_1, \dots, g_{n-1}$ are in the radical.
      We want to show that $\phi_{n}$ is then already an isomorphism, thus reducing the case of $\phi_0$ and $\phi_{n+1}$ being isomorphisms to the already solved case of $\phi_n$ and $\phi_{n+1}$ being isomorphisms.

      To obtain an inverse of $\phi_n$ we first show $g_n f_{n+1} = 0$.
      Following the idea of the last part of the proof in \cite[proposition 4.8]{jasso_n-abelian_2016} we look at \cref{figure:n+2-angle-of-zero-morphism}. 
      Since $\sA$ is $n$-extension closed, we can find the dashed maps $h_0, \dots, h_n$ of \cref{figure:n+2-angle-of-zero-morphism}
      \begin{figure}[h]
         \includepdf{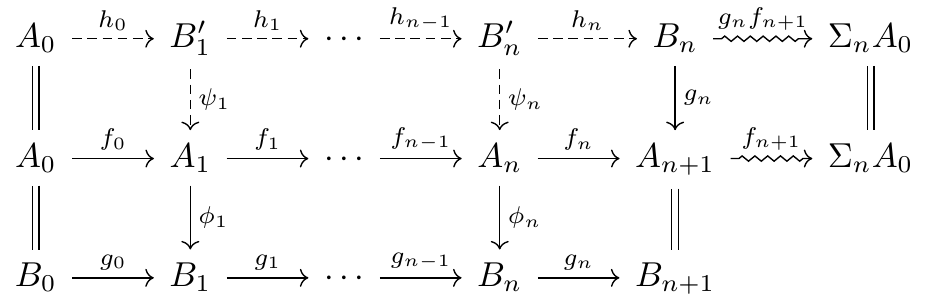}
         \caption{An $(n+2)$-angle arising from $g_n f_{n+1}$.}
         \label{figure:n+2-angle-of-zero-morphism}
      \end{figure}
      such that the upper row is an $(n+2)$-angle and $B'_1,\dots,B'_n \in \sA$.
      By the axiom (F3) from \cite[definition 2.1]{geiss_n-angulated_2013} we find $\psi_1, \dots, \psi_n$ such that \cref{figure:n+2-angle-of-zero-morphism} is a commutative diagram. 
      However, composition of the vertical morphisms in \cref{figure:n+2-angle-of-zero-morphism} yields that the diagram 
      \includepdf{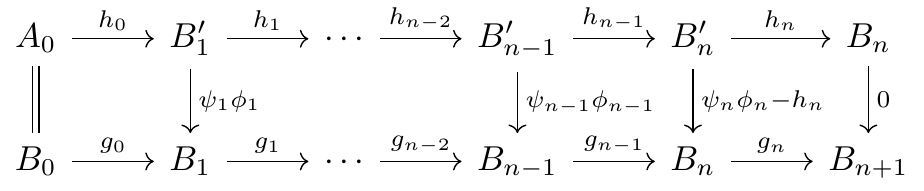}
      is commutative, using $h_{n-1} h_n = 0$ for commutativity of the penultimate square.
      This chain map is homotopic to the zero chain map by the dual of \cite[lemma 2.1]{jasso_n-abelian_2016}. In particular, this shows that $h_0$ is a split-mono and hence $h_n$ is split-epi by \cite[proposition 2.6]{jasso_n-abelian_2016}.
      Now, this yields $g_n f_{n+1} = 0$.

      Since $(n+2)$-angles are exact, $g_n f_{n+1} = 0$ proves existence of a morphism $\phi'_n \colon B_n \to A_n$ as in \cref{figure:weak-equivalence-identity} 
      \begin{figure}[h]
         \includepdf{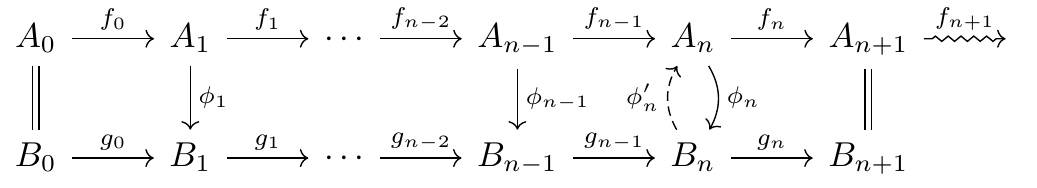}
         \caption{Construction of $\phi'_n$.}
         \label{figure:weak-equivalence-identity}
      \end{figure}
      satisfying $\phi'_n f_n = g_n$.
      From this we obtain $(\phi_n \phi'_n -\id_{A_{n+1}}) f_n = 0$.
      Yet again, by exactness of $(n+2)$-angles, the morphism $\phi_n \phi'_n -\id_{A_{n+1}}$ factors through the radical morphism $f_{n-1}$.
      Therefore, $\phi_n \phi'_n$ is an isomorphism.
      Similarly, $\phi'_n \phi_n - \id_{B_{n+1}}$ factors through $g_{n-1}$ using that $E'$ is an $n$-exact sequence in $\sA$.
      As $g_{n-1}$ is in the radical as well, $\phi'_n \phi_n$ is an isomorphism.
      This shows that $\phi_n$ is an isomorphism.
      Since $\phi_{n+1}$ is an isomorphism as well, $E'$ is an $\sA$-conflation by the first part of the proof.
   \end{proof}

   The axiom (E0) of \cite[definition 4.2]{jasso_n-abelian_2016} is trivially satisfied:

   \begin{lemma} \label{lemma:E0-axiom}
      The $n$-exact sequence $0 \to 0 \to \cdots \to 0$ is an $\sA$-conflation.
   \end{lemma}
   \begin{proof}
      This follows from $0 \in \sA$ as $\sA \subset \sF$ is additive.
   \end{proof}
   
   That the axiom (E1) and (E1${}^{\op}$) of \cite[defintion 4.2]{jasso_n-abelian_2016} are satisfied for $(\sA, \sE_{\sA})$ follows almost immediately from our work in \cref{subsection:about_N4}.
   We only present a proof of (E1) as the proof of (E1${}^{\op}$) is completely analogous.
   
   \begin{lemma} \label{lemma:E1-axiom}
      The composite $h_0 = g_0 f_0$ of two $\sA$-inflations $f_0$ and $g_0$ is an $\sA$-inflation.
   \end{lemma}
   \begin{proof}
      Let $f_0 \colon X_0 \to X_0'$ and $g_0 \colon X_0' \to X_1'$ be $\sA$-inflations and $h_0 = f_0 g_0 \colon X_0 \to X_1'$ their composite.
      Because $f_0$ and $g_0$ are $\sA$-inflations we can choose $X_2, \dots, X_{n+1} \in \sA$ and $X_2', \dots, X'_{n+1} \in \sA$, as well as $X''_2, \dots, X''_{n+1} \in \sF$ so that all three rows in \cref{figure:N4-asterik-axiom-1} are $(n+2)$-angles and the middle row in \cref{figure:N4-asterik-axiom-1} is a minimal $(n+2)$-angle.
      In the notation of \cref{figure:N4-asterik-axiom-1} and \cref{lemma:comparison-lemma}, using that $\sA$ is $n$-extension closed and that $g_{n+1} \Sigma_n f_1$ is a morphism from an object in $\sA$ to an object in $\Sigma_n \sA$, we see that $Y_1, \dots, Y_n$ belong to $\sA$ as the minimal $(n+2)$-angle of $g_{n+1} \Sigma_n f_1$ depicted in \cref{lemma:comparison-lemma} is a direct summand of any $(n+2)$-angle with $g_{n+1} \Sigma_n f_1$ in position $n+2$.
      However, this means that the objects $X''_2, \dots, X''_{n+1}$ of the minimal $(n+2)$-angle of $h_0$ belong to $\sA$ using (\ref{enum:begin-Z-add-property})-(\ref{enum:end-sum-property}) of \cref{lemma:comparison-lemma}.
      Hence $h_0$ is an $\sA$-inflation.
   \end{proof}

   Finally, that (E2) and (E2${}^{\op}$) of \cite[definition 4.2]{jasso_n-abelian_2016} are satisfied for $(\sA, \sE_{\sA})$ follows easily from the definitions.
   Again, we present a proof for (E2) only, as the proof of (E2${}^{\op}$) is completely analogous.

   \begin{lemma} \label{lemma:E2-axiom}
      Suppose we are given a diagram
      \includepdf{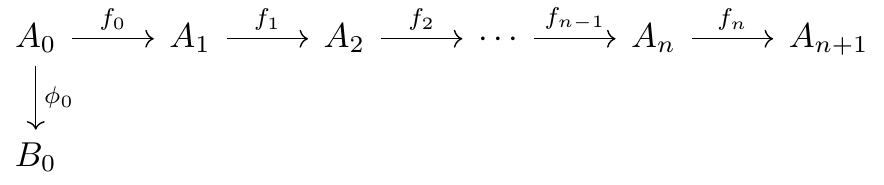}
      where the upper row is a $\sA$-conflation $A$. 
      Then there is an $n$-pushout diagram \cref{figure:n-pushout}
      \begin{figure}[h]
      \includepdf{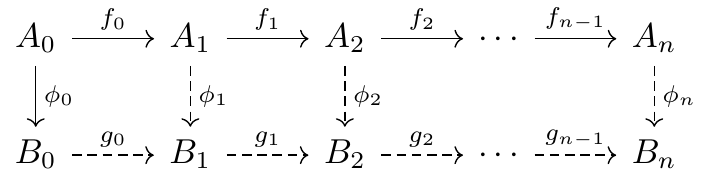}
      \caption{$n$-Pushout diagram of $A$ along $\phi_0$.}
      \label{figure:n-pushout}
      \end{figure}
      in the sense of \cite[definition 2.11]{jasso_n-abelian_2016} such that $g_0$ is an $\sA$-inflation.
   \end{lemma}
   \begin{proof}
      Let $A$ be a $\sA$-conflation as in the lemma.
      This means that we find a morphism $f_{-1} \colon \Sigma_n^{-1} A_{n+1} \to A_0$ completing $A$ to an $(n+2)$-angle.
      This is to say, we are given the undashed morphisms of \cref{figure:E2-axiom-2}
      \begin{figure}[h]
         \includepdf{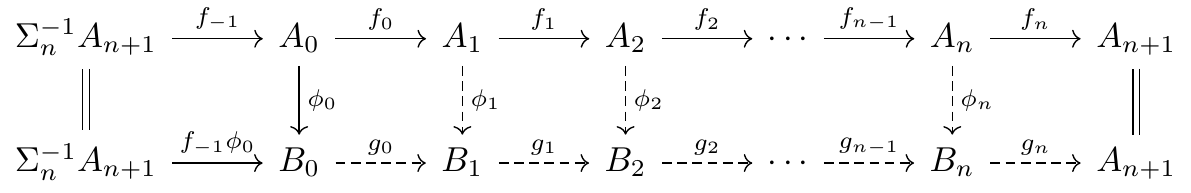}
         \caption{Construction of an $n$-pushout of $A$ along $\phi_0$.}
         \label{figure:E2-axiom-2}
      \end{figure}
      so that the upper row is an $(n+2)$-angle.
      Since $\sA \subset \sF$ is an $n$-extension closed subcategory, we find objects $B_1, \dots, B_n \in \sA$ and the dashed morphisms $g_0, \dots, g_n$ as in \cref{figure:E2-axiom-2} such that the lower row is an $(n+2)$-angle.
      Notice, $g_0$ is an $\sA$-inflation, as the second row of \cref{figure:E2-axiom-2} is an $(n+2)$-angle.
      By the axiom (F4) of \cite[defintion 2.1]{geiss_n-angulated_2013} there are dashed morphisms $\phi_1, \dots, \phi_n$ in \cref{figure:E2-axiom-2} so that the depicted morphism $\phi = (\phi_0, \dots, \phi_n, \id_{A_{n+1}})$ of $(n+2)$-angles is a good morphism of $(n+2)$-angles, i.e.\ so that the mapping cone of $\phi$ is an $(n+2)$-angle as well.
      Then similarly to \cite[lemma 4.1]{bergh_axioms_2013} we conclude that the $n$-$\Sigma_n$-sequence $P$ shown in \cref{figure:E2-axiom-direct-summand}
      \begin{figure}[h]
         \includepdf{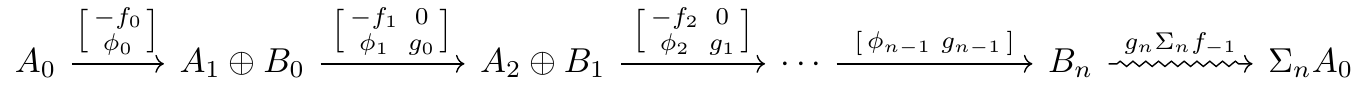}
         \caption{Direct summand $P$ of the mapping cone of $\phi$.}
         \label{figure:E2-axiom-direct-summand}
      \end{figure}
      is an $(n+2)$-angle, as $\pentagon$ is closed under direct summands and because $P$ is a direct summand of the cone of $\phi$, as shown in
      \begin{figure}[h]
         \centering \includegraphics{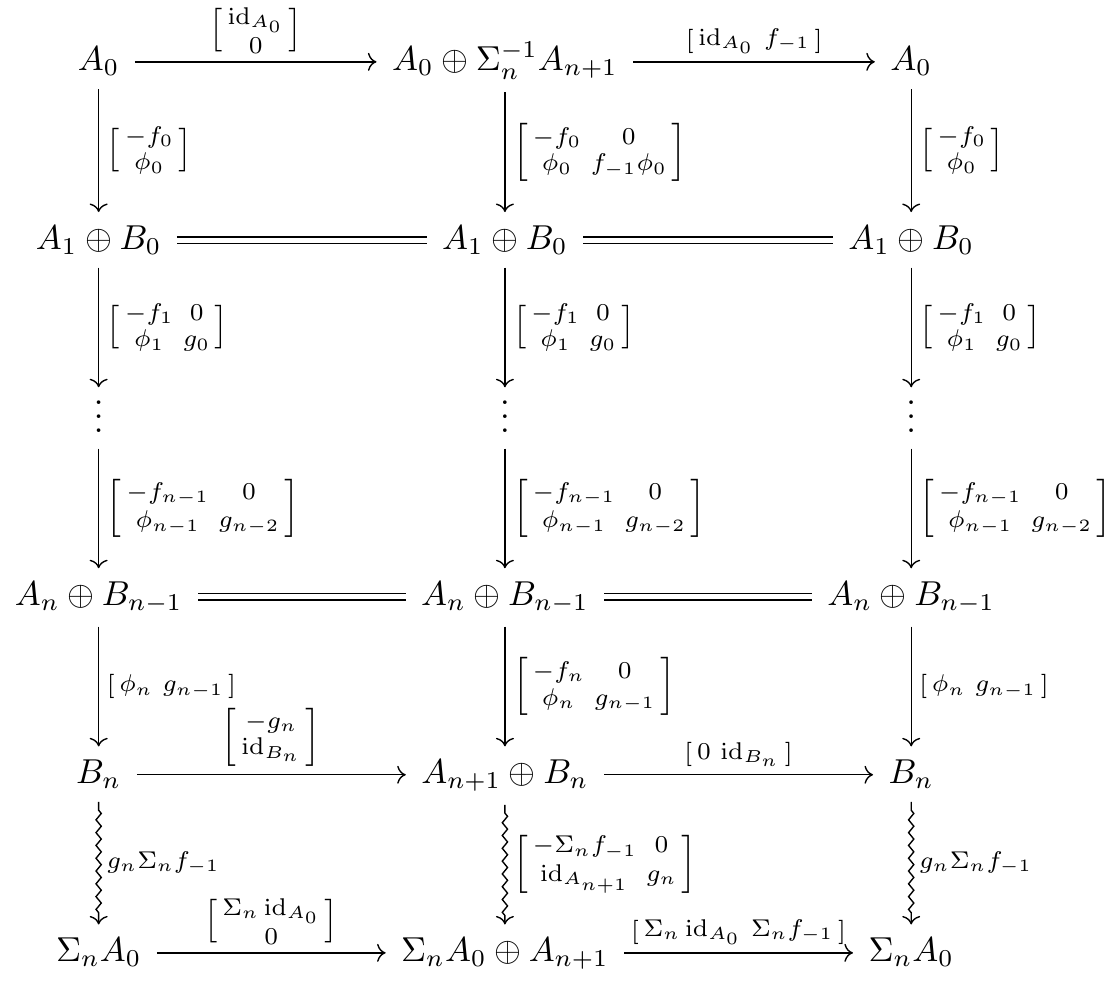}
         \caption{Diagram showing that the $n$-$\Sigma_n$-sequence $P$ given in \cref{figure:E2-axiom-direct-summand} is a direct summand of the mapping cone of $\phi = (\phi_0, \dots, \phi_n, \id_{A_{n+1}})$ as shown in \cref{figure:E2-axiom-2}.}
         \label{figure:E2-axiom-direct-summand-2}
      \end{figure}
      \cref{figure:E2-axiom-direct-summand-2}.
      Moreover, since all $\sA$-conflations are $n$-exact by \cref{lemma:A-conflations-are-n-exact}, we conclude from the $(n+2)$-angle $P$ shown in \cref{figure:E2-axiom-direct-summand} that the constructed objects $B_1, \dots, B_n$ and morphisms $g_0, \dots, g_{n-1}$ and $\phi_1, \dots, \phi_{n}$ make \cref{figure:n-pushout} an $n$-pushout diagram.
   \end{proof}

   Notice, \cref{lemma:A-conflations-are-n-exact}, \cref{lemma:weak-isomorphism-closure}, \cref{lemma:E0-axiom}, \cref{lemma:E1-axiom} and \cref{lemma:E2-axiom} show that part (\ref{enum:maintheorem-1}) of \cref{theorem:maintheorem} holds.
   It remains to show part (\ref{enum:maintheorem-2}) of \cref{theorem:maintheorem}.
   In the following we will write $[E]$ for the equivalence class of $E$ in $\YExt_{(\sA, \sE_{\sA})}^n(A_{n+1}, A_0)$, where a conflation $E \in \sE_{\sA}$ is viewed as an $n$-extension of an object $A_{n+1} \in \sA$ by an object $A_{0} \in \sA$.

   \begin{lemma} \label{lemma:Yext-isomorphism}
      There is a natural bilinear isomorphism 
      \[\Hom_{\sF}(-, \Sigma_n(-)) \to \YExt^n_{(\sA, \sE_{\sA})}(-, -)\] 
      of functors $\sA \times \sA^{\op} \to \Ab$.
   \end{lemma}
   \begin{proof}
      We construct $\Psi_{A_0,A_{n+1}} \colon \YExt_{(\sA, \sE_{\sA})}^n(A_{n+1}, A_0) \to \Hom_{\sF}(A_{n+1}, \Sigma_n A_0)$ for any fixed pair $A_0, A_{n+1} \in \sA$:
      Suppose $E \in \sE_{\sA}$ is an $\sA$-conflation with end terms $A_0$ and $A_{n+1}$.
      Then there is a $(n+2)$-angle $X_E$ of the shape
      \includepdf{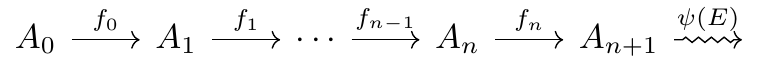}
      completing $E$ to an $(n+2)$-angle.
      Notice that $\psi(E)$ is uniquely determined by $E$, using \cref{lemma:connecting-morphism-is-unique}.
      Hence, this way a map $\psi \colon \sE_{\sA} \to \Hom_{\sF}(\sA, \Sigma_n \sA)$ is defined.
      Moreover, \cref{lemma:connecting-morphism-is-unique} shows that we have $\psi(E) = \psi(E')$ if $E$ and $E'$ are linked by a sequence of equivalences in the sense of \cite[defintion 2.9]{jasso_n-abelian_2016}, i.e.\ if $[E] = [E']$.
      Hence, $\psi$ induces a map $\Psi_{A_0, A_{n+1}} \colon \YExt_{(\sA, \sE_{\sA})}^n(A_{n+1}, A_0) \to \Hom_{\sF}(A_{n+1}, \Sigma_n A_0), \, [E] \mapsto \psi(E)$ for each pair $A_0,A_{n+1} \in \sA$.

      Conversely we construct $\Phi_{A_0,A_{n+1}} \colon \Hom_\sF(A_{n+1}, \Sigma_n A_0) \to \YExt^n_{(\sA, \sE_{\sA})}(A_{n+1}, A_0)$ for each fixed pair $A_0, A_{n+1} \in \sA$: 
      Assume we are given an $f \in \Hom_{\sF}(A_{n+1}, \Sigma_n A_0)$.
      Since $\sA$ is $n$-extension closed, we can construct an $(n+2)$-angle $X_f$ of the shape in \cref{figure:completion-of-f}
      \begin{figure}[h]
         \includepdf{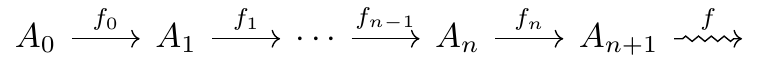}
         \caption{Completion of $f$ to an $(n+2)$-angle.}
         \label{figure:completion-of-f}
      \end{figure}
      with $A_1, \dots, A_{n} \in \sA$. 
      Denote the $\sA$-conflation arising from such an $(n+2)$-angle $X_f$ by $E_{X_f}$.
      We claim, the equivalence class $[E_{X_f}]$ does not depend on the choice of $X_f$:
      Suppose we are given another completion of $f$ to an $(n+2)$-angle $X'_f$ of the shape
      \includepdf{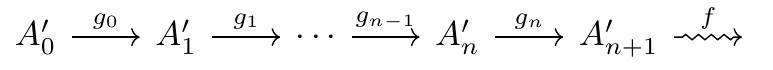}
      with $A'_1, \dots, A'_{n} \in \sA$.
      Using the axiom (F3) from \cite[definition 2.1]{geiss_n-angulated_2013} we find the dashed morphisms of a commutative diagram
      \includepdf{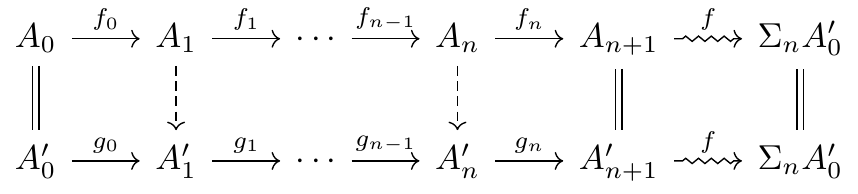}
      and hence obtain $[E_{X_f}] = [E_{X'_f}]$.
      Therefore, for $A_0,A_{n+1} \in \sA$ the assignment 
      \[\Phi_{A_0,A_{n+1}} \colon \Hom_\sF(A_{n+1}, \Sigma_n A_0) \to \YExt^n_{(\sA, \sE_{\sA})}(A_{n+1}, A_0), \, f \mapsto [E_{X_f}]\] 
      where $X_f$ is an arbitrary $(n+2)$-angle as in \cref{figure:completion-of-f} is well-defined.
      By construction it is clear that $\Phi_{A_0,A_{n+1}}$ and $\Psi_{A_0, A_{n+1}}$ for any fixed pair $A_0,A_{n+1} \in \sA$ are mutually inverse to each other.
      It is straight forward to check, that this way a natural and bilinear isomorphism $\Phi$ is defined.
   \end{proof}

   \begin{acknowledgment}
      Thanks to Peter Jørgensen for suggesting this project to me and for his mathematical and linguistic advice and corrections.

      This work was supported by Aarhus University Research Foundation, grant no. AUFF-F-2020-7-16.
   \end{acknowledgment}

\bibliographystyle{alpha}

\begin{thebibliography}{GKO13}

\bibitem[BBD82]{beilinson_faisceau_pervers_82} A.\ A.\ Beilinson, J.\ Bernstein, and P.\ Deligne, {\it Faisceaux pervers}, Ast\'{e}risque {\bf 100} (1982) (proceedings of the conference ``Analysis and topology on singular spaces'', Vol.\ 1, Luminy, 1981).

\bibitem[BT13]{bergh_axioms_2013}
Petter~Andreas Bergh and Marius Thaule.
\newblock {The axioms for $n$–angulated categories}.
\newblock {\em Algebraic \& Geometric Topology}, 13(4):2405 -- 2428, 2013.

\bibitem[Dye05]{dyer_exact_2005}
Matthew~J Dyer.
\newblock Exact subcategories of triangulated categories.
\newblock \url{https://www3.nd.edu/~dyer/papers/extri.pdf}, preprint (2005).

\bibitem[Fed19]{Fedele_2019}
Francesca Fedele.
\newblock Auslander–reiten $(d+2)$-angles in subcategories and a
  $(d+2)$-angulated generalisation of a theorem by brüning.
\newblock {\em Journal of Pure and Applied Algebra}, 223(8):3554–3580, Aug
  2019.

\bibitem[GKO13]{geiss_n-angulated_2013}
Christof Geiss, Bernhard Keller, and Steffen Oppermann.
\newblock $n$-{angulated} categories.
\newblock {\em Journal für die reine und angewandte Mathematik (Crelles
  Journal)}, 2013(675):101--120, January 2013.

\bibitem[Jas16]{jasso_n-abelian_2016}
Gustavo Jasso.
\newblock $n$-{Abelian} and $n$-exact categories.
\newblock {\em Mathematische Zeitschrift}, 283(3-4):703--759, August 2016.

\bibitem[Jør20]{joergensen_simple_minded_2020} 
Peter Jørgensen. 
\newblock Abelian subcategories of triangulated categories induced by simple minded systems. 
\newblock \url{https://drive.google.com/file/d/1gdM_isAsEgltPKEygsvlXDW9VZmHIDWq/view}, preprint (2020).

\bibitem[OT12]{oppermann_higher-dimensional_2012}
Steffen Oppermann and Hugh Thomas.
\newblock Higher-dimensional cluster combinatorics and representation theory.
\newblock {\em Journal of the European Mathematical Society}, 14(6):1679--1737,
  2012.
\end{thebibliography}

\end{document}